\documentclass[10pt,a4paper]{article}
\usepackage{authblk}
\usepackage{amsmath}
\usepackage{amssymb}
\usepackage{amsthm}

\title{Coherent ultrafilters and nonhomogeneity}
\author{Jan Star\'y}
\affil{Czech Technical University, Prague}
\date{}

\theoremstyle{definition}
\newtheorem{definition}{Definition}[section]
\newtheorem{example}[definition]{Example}
\newtheorem{question}[definition]{Question}

\theoremstyle{plain}

\newtheorem{observation}[definition]{Observation}
\newtheorem{lemma}[definition]{Lemma}
\newtheorem{proposition}[definition]{Proposition}

\newtheorem{corollary}[definition]{Corollary}

\def\B{{\mathcal B}}
\def\C{{\mathcal C}}

\def\F{{\mathcal F}}

\def\J{{\mathcal J}}

\def\M{{\mathcal M}}

\def\U{{\mathcal U}}

\def\c{{\mathfrak c}}
\def\d{{\mathfrak d}}

\def\phi{\varphi}

\let\sub = \subseteq

\let\join = \vee
\let\Join = \bigvee
\let\meet = \wedge
\let\Meet = \bigwedge
\let\Union = \bigcup

\newcommand{\cov}{\operatorname{cov}}

\newcommand{\St}{\operatorname{St}}

\def\setof#1#2{\left\{#1; #2\right\}}
\def\seqof#1#2{\left(#1; #2\right)}

\begin{document}
\maketitle

\begin{abstract}
We introduce the notion of a {\em coherent $P$-ultrafilter\/}
on a complete ccc Boolean algebra, strenghtening the notion
of a $P$-point on $\omega$,
and show that these ultrafilters exist generically under $\c = \d$.
This improves the known existence result of Ketonen \cite{K}.
Similarly, the existence theorem of Canjar \cite{C} can be extended
to show that {\em coherently selective ultrafilters\/}
exist generically under $\c = \cov{\M}$.

We use these ultrafilters in a topological application:
a coherent $P$-ultrafilter on an algebra $\B$
is an {\em untouchable point\/} in the Stone space of $\B$,
witnessing its nonhomogeneity.
\end{abstract}



\section{Introduction}

The article is organized as follows.

In section \ref{partitions}, we describe the lattice $Part(\B)$
of partitions of a complete ccc Boolean algebra $\B$
and see how a given ultrafilter $\U$ on $\B$
interplays with this lattice.

In section \ref{coherent}, we define {\em coherent $P$-ultrafilters\/}
and {\em coherently selective ultrafilters\/} on a complete ccc algebra
and show that they exist generically, i.e. every filter with a small base
can be extended into such ultrafilter, under conditions isolated
in \cite{K} and \cite{C}.

In section \ref{homogeneity}, we recall the homogeneity problem
for extremally disconnected compact Hausdorff spaces
--- the Stone spaces of complete Boolean algebras.
We show the relevance of coherent ultrafilters to this question:
a coherent $P$-ultrafilter on a complete ccc Boolean algebra
is an {\em untouchable point\/} in the corresponding Stone space.

\section{The lattice of partitions}\label{partitions}

Recall that a {\em partition\/} of a Boolean agebra
is a maximal antichain. We will denote the set of all
infinite partitions of an algebra $\B$ by $Part(\B)$.
We only consider infinite algebras.

\begin{definition}
Let $\B$ be a Boolean algebra. For two partitions $P,Q$ of $\B$,
say that $P$ {\it refines\/} $Q$ and write $P\preceq Q$
if for each $p\in P$ there is exactly one $q\in Q$ such that $p\leq q$.
Call $P\meet Q = \setof{p\meet q}{p\in P, q\in Q} \setminus \{0\}$
the {\it common refinement\/} of $P$ and $Q$.
\end{definition}

The relation $P\preceq Q$ is easily seen to be a partial order on $Part(\B)$.
Note that $P\meet Q$ is indeed a partition of $\B$
that refines both $P$ and $Q$. In fact, it is the
infimum of $\{P, Q\}$ in $(Part(\B), \preceq)$,
and makes $\left(Part(\B), \meet, \{1_\B\}, \preceq\right)$
a~semilattice with unit.

\begin{observation}
For a complete Boolean algebra $\B$,
the order $(Part(\B), \preceq)$ is a lattice.
This lattice is complete if and only if $\B$ is atomic.
\end{observation}

\begin{proof}
We show that, in fact, every system
$\setof{P_\alpha}{\alpha\in\kappa}\sub Part(\B)$ has a supremum.
Fix any $P$ from the system.
For $p\in P$, put $p_0 = p$ and inductively define
$$q_n = \Join\setof{q\in \Union{P_\alpha}}{q \parallel p_n}$$
$$p_{n+1} = \Join\setof{p\in P}{p \parallel q_n}.$$

It is clear that $p\leq p_n\leq q_n \leq p_{n+1} \leq q_{n+1}$
for each $n\in\omega$. Put $u(p) = \Join\setof{p_n}{n\in\omega}
= \Join\setof{q_n}{n\in\omega}$.
It is easily seen that the set $\Join{P_\alpha} = \setof{u_p}{p\in P}$
does not depend on the choice of the starting partition $P$.
Clearly, $\Join{P_\alpha}$ is a partition refined by every $P_\alpha$;
we show that it is the finest among such partitions,
and therefore a supremum of $\setof{P_\alpha}{\alpha\in\kappa}$.

Let $P_\alpha\preceq R$ for every $\alpha\in\kappa$.
It suffices to see that whenever $p\leq r$ for
some $p\in P_\alpha$ and $r\in R$, we also have $u(p) \leq r$;
this can be shown by induction for every $p_n, q_n$
as defined above. Let $p\in P$ and let $r$ be the only member of $R$
such that $p\leq r$. Every $q\in \Union{P_\alpha}$ is below exactly one
$r^\prime\in R$, and if $r\not=r^\prime$, then $q\perp p$; hence
$q_0 = \Join\setof{q\in \Union{P_\alpha}}{q\parallel p} \leq r$. Similarly,
$p_1 = \Join\setof{p\in P}{p\parallel q_0} \leq r$, and it follows
by induction that every $p_n\leq q_n \leq r$.
Hence $u(p) \leq r$ and $\Join{P_\alpha}\preceq R$.

For completeness,
let $\setof{P_\alpha}{\alpha\in\kappa}$ be a system of partitions.
A supremum $\Join{P_\alpha}$ exists in $Part(\B)$ by the above.
A complete atomic algebra is a~powerset algebra, which is completely
distributive. The partition $P =
\setof{\Meet_{\alpha\in\kappa}{f(\alpha)}}{f\in\prod{P_\alpha}} \setminus \{0\}$
is easily seen to be the infimum of the $P_\alpha$. In particular, the set
of all atoms is the finest partition of $\B$, i.e., the smallest element
of $Part(\B)$.
In the other direction, if $(Part(\B), \preceq)$ is complete, it must have
a~smallest element, which clearly needs to be a partition consisting
exclusively of atoms of $\B$.
\end{proof}

Note that for atomless algebras,
completeness is actually necessary in the previous observation:
the following example shows that in an atomless algebra $\B$
that is not $\sigma$-complete,
two partitions can always be found that do not have a supremum
in $Part(\B)$.

\begin{example}
Let $A = \setof{a_n}{n\in\omega}\sub\B$ be
a countable subset without a supremum in $\B$;
without loss of generality, $A$ is an antichain.
Let $\C$ be the completion of $\B$,
and consider $c = \Join^\C{A} \in \C\setminus\B$.
The element $-c\in\C$ can be partitioned into some
$\setof{x_\alpha}{\alpha\in\kappa} = X \sub\B$,
as $\B$ is dense in $\C$.

Split every $a_n\in A$ into $a_n^0\join a_n^1$,
put $b_0 = a_0^0$, $b_{n+1} = a_n^1\join a_{n+1}^0$
and $B = \setof{b_n}{n\in\omega}$.
Then clearly $\Join^\C{B} = \Join^\C{A} = c$.
Put $P = A \cup X, Q = B \cup X$.
Now $P,Q$ are partitions of $\B$,
and we show that $\{P,Q\}$ has no supremum in $Part(\B)$.

Let $R\in Part(\B)$ satisfy $P,Q\preceq R$.
Then there must be some $r\in R$ such that $r\geq a_n, b_n$ for all $n$;
but $r\in\B$ cannot be a supremum of $a_n$,
hence $r$ meets some $x\in X$. In fact, we have $x\leq r$,
as $X \sub P\cap Q$ and $P,Q\preceq R$.
Then the partition $R_0$ which contains $r-x, x \in R_0$ instead of $r\in R$
satisfies $P,Q\preceq R_0 \prec R$. Hence $R$ is not a supremum.
\end{example}

Recall that partitions $P,Q\in Part(\B)$ are {\it independent\/}
if $p\meet q \ne 0$ for every $p\in P$ and $q\in Q$. More generally,
$\setof{P_i}{i\in I}$ is an {\it independent system of partitions\/}
if for every finite $K\sub I$ and every $f\in\Pi\setof{P_i}{i\in K}$,
the intersection $\Meet\setof{f(i)}{i\in K}$ is nonzero.

Note that if $P,Q$ are independent,
then $P\join Q = \{1_B\}$ in $Part(\B)$.

%\begin{definition}
%A {\it partition filter\/} is a filter in $(Part(\B), \preceq)$; i.e.,
%a subset $F\sub Part(\B)$ that contains with every $P\in F$ all $Q\in Part(\B)$
%such that $P\preceq Q$, and contains $P\meet Q$ with every $P,Q\in F$. Clearly,
%$Part_{fin}(\B)$ is a partition filter.
%\end{definition}

\subsection{The structure induced by partitions}

Let $\B$ be a complete ccc Boolean algebra. For $P\in Part(\B)$,
let $\B_P$ be the subalgebra completely generated by $P\sub\B$.
Denote the inclusion as $e_P:\B_P\sub \B$.
If $P\preceq Q$, let $e_P^Q$ be the inclusion of $\B_Q$ in $\B_P$.
The family $\setof{\B_P}{P\in Part(\B)}$ together with the mappings
$e^P_Q$ forms a directed system of complete Boolean algebras
indexed by the directed set $(Part(\B), \succeq)$.

%\medskip
%The restriction to complete ccc algebras is not strictly necessary
%in the above definition; a more general situation could be described,
%minding the possible size of partitions and suitable $\kappa$-completeness
%of the algebra.
%It is however the complete ccc case which we are mostly interested in.

\begin{observation}
In the setting described above,
\begin{itemize}
\item [\rm(a)]
for each $P\in Part(\B)$,
the algebra $\B_P$ is isomorphic to $P(\omega)$;
\item [\rm(b)]
$\B_{P\meet Q}$ is completely generated by $\B_P\cup\B_Q$,
and $\B_{P\join Q} = \B_P\cap\B_Q$;
\item [\rm(c)]
$\B_P\cap\B_Q = \{0_\B, 1_\B\}$ iff $P\join Q = \{1_\B\}$;
%\item
%The lattice $(Part(\B), \preceq)$ embedds
%into the latice of subalgebras of $\B$ ordered by reverse inclusion.
%Specificaly, $Part(\B)$ is isomorphic to the sublattice consisting
%of copies of $P(\omega)$.
\item [\rm(d)]
for $P\preceq Q$, the embedding $e_P^Q: \B_Q \sub \B_P$ is regular;
\item [\rm(e)]
for each $P\in Part(\B)$, the embedding $e_P: \B_P \sub \B$ is regular.
\end{itemize}
\end{observation}

\begin{lemma}
The algebra $\B$, with the regular embeddings $e_P:\B_P\to\B$,
is a direct limit of the directed system
of algebras $\B_P$ and mappings $e_P^Q$.
%In fact, $\B$ is a limit of every subsystem
%consisting of $\B_P$ and $e_P^Q$ for $P,Q\notin F$,
%where $F\sub Part(\B)$ is a partition filter.
\end{lemma}

%% FIXME: odkazujeme se na chapter I z disertace
% on je ten dukaz ale dost jasny.
%\begin{proof}
%Every triangle commutes,
%i.e. $e_P\circ e_P^Q = e_Q$ whenever $P\preceq Q$.
%The algebra $\B$ is easily seen to be isomorphic
%to the direct limit as described in Chapter I.
%Put $\phi(x) = [x]_\approx$ for $x\in\B$.
%Then $\phi:\B\to (\bigsqcup\B_P/\approx)$
%is well defined; in fact, the equivalence relation $x \approx y$ iff
%$x\in\B_P, y\in\B_Q, e_{P\meet Q}^P(x) = e_{P\meet Q}^Q(y)$ reduces
%to $x = y$ in $\B$, and merely factorizes out the formal distinction
%between multiple copies of $x\in\B$ coming from different components
%$\B_P$ of the disjoint union; hence $\phi$ is one-to-one, too.
%Clearly, $\phi$ is onto, and it is easily checked to be homomorphic.
%The second part follows from the fact that for every partition filter $F$,
%the system formed by $\B_P$ and $e_P^Q$ for $P,Q\notin F$
%is a cofinal directed subsystem.
%\end{proof}

% FIXME
%Let $\C$ be a (maximal?) chain in $(Part(\B), \leq)$.
%Then $X_P = St(\B_P)$ forms a chain of copies of $\beta\omega$,
%where $P\preceq Q\in \C$ implies $X_Q$ is an image of $X_P$.
%\rem{to spis pouzijeme u faktoru B/P, kterym odpovidaji uzavrene podprostory}

For $P\in Part(\B)$, let $\J_P$ be the ideal on $\B$ generated by $P\sub\B$.
Note that $\J_{P \meet Q} = \J_P \cap \J_Q$
and $\J_P \sub \J_Q$ for $P\preceq Q$.
Write $\B/P$ for $\B/\J_P$ and $\B_P/P$ for $\B_P/\J_P$.
Whenever $P\preceq Q\in Part(\B)$, we have $\J_P\sub\J_Q$,
hence the algebra $\B/Q$ is a quotient of $\B/P$;
denote the quotient mapping by $f_P^Q:\B/P \to \B/Q$.
The family of algebras ${\B/P}$ and mappings $f_P^Q$ for $P,Q\in Part(\B)$
forms an inverse system indexed by $(Part(\B), \succeq)$.

% FIXME: je to k necemu?
% je to nejaky separativni kvocient toho \preceq nebo neco?
%\begin{definition}
%For $P, Q\in Part(\B)$, let $P\approx Q$ if $\J_P = \J_Q$
%\end{definition}
%
%\begin{observation}
%\begin{itemize}
%\item
%If $P\preceq Q$ finitely, then $P\approx Q$.
%\item
%If $P\approx Q$, then $P\approx P\meet Q$.
%\item
%If $P\approx R$ and $P\preceq Q \preceq R$, then $P\approx Q$.
%\end{itemize}
%\end{observation}
%
%\begin{proof}
%(2) Clearly $\J_{P\meet Q} \sub \J_P$. Every $p\in P$ is covered by finitely
%many $q_i\in Q$, hence also by finitely many $p\meet q_i\in P\meet Q$.
%Thus also $\J_P \sub \J_{P\meet Q}$.
%(3) $\J_P\sub\J_Q$ is immediate from $P\preceq Q$. Every $q\in Q$ is covered
%by one $r\in R$, which in turn is covered by finitely many $p_i\in P$. 
%Hence $\J_Q\sub\J_P$, too.
%\end{proof}

\begin{observation}
In the setting described above,
\begin{itemize}
\item [\rm(a)]
for each $P\in Part(\B)$, the quotient $\B_P/P$
is isomorphic to $P(\omega) / fin$;
\item [\rm(b)]
the inclusion $\B_P/P \sub \B/P$ is a regular embedding.
%\item [\rm()]
%Whenever $P\preceq Q$, we have $\B_P/P \bus \B_Q/P$;
%if $P\prec Q$, then $\B_Q/P \iso \B_Q \iso P(\omega)$
%\rem{coz nemuze byt regular subalgebra of $\B_P/P \iso P(\omega)/fin$.
%- co je regularizacni ideal?}
%\item [\rm()]
%Whenever $P\preceq Q$, we have $P(\omega)/fin \iso \B_Q/Q \sub \B_P/Q$.
%\rem{Co je to za algebru $B_P/Q$? Ona je faktorem $P(\omega)/fin \iso \B_P/P$,
%a zaroven ma $\B_Q/Q \iso P(\omega)/fin$ jakozto podalgebru.}
\end{itemize}
\end{observation}

\begin{lemma}
The algebra $\B$, with the quotient mappings $f_P:\B\to\B/P$,
is an inverse limit of the inverse system
of algebras $\B/P$ and mappings $f_P^Q$.
\end{lemma}

Employing the Stone duality, we can summarize that

\begin{corollary}
\begin{itemize}
\item [\rm(a)]
Every infinite complete ccc algebra is a limit
of a directed system of copies of $P(\omega)$.
Dually, every infinite ccc extremally disconnected compact space
is an inverse limit of an inverse system of copies of $\beta\omega$.
\item [\rm(b)]
Every infinite complete ccc Boolean algebra is an inverse limit
of an inverse system of copies of $P(\omega)/fin$.
Dually, every infinite ccc extremally disconnected compact space
is a direct limit of a system of copies of $\omega^*$.
\end{itemize}
\end{corollary}

% FIXME
%\rem{
%$B/P$ muze mit netrivialni automorfismy, tj: nejsou nesene
%zadnym automorfismem $B_P$. V jakem vztahu jsou automorfismy
%$\B$ a automorfismy $\B/P$? }

%\subsection{Ultrafilters and the partition structure}

For an ultrafilter $\U$ on $\B$ and $P$ a partition of $\B$,
let $\U_P = \U \cap \B_P$, which is clearly an ultrafilter on $\B_P$.
As $\B_P$ is isomorphic to $P(\omega)$, the ultrafilter $\U_P$
can be viewed as an ultrafilter on $\omega$.

\begin{observation}
Let $\B$ be a complete atomless ccc algebra,
let $P,Q$ be partitions of $\B$,
and let $\U$ be an ultrafilter on $\B$.
Then
\begin{itemize}
\item [\rm(a)]
$P\cap\U \not=\emptyset$ if and only if $\U_P$ is trivial.
%\item [\rm(b)]
%$\setof{P\in Part(\B)}{\U \cap P \not= \emptyset}$
%is a proper partition filter if $\U$ is nontrivial.
\item [\rm(b)]
$\setof{P\in Part(\B)}{\U \cap P = \emptyset}$
is an open dense subset of $(Part(\B), \preceq)$.
\item [\rm(c)]
$\U_Q = \U_P \cap \B_Q$ for $P\preceq Q$.
\item [\rm(d)]
$\B = \Union\setof{\B_P}{P\cap\U=\emptyset}$
\end{itemize}
\end{observation}

\section{Coherent ultrafilters}\label{coherent}

\begin{definition}
Let $\B$ be a complete, atomless, ccc algebra. For a property $\phi$
of families of subsets of $\omega$, we say that a subset $X\sub\B$ is
a {\it coherent $\phi$-family\/} on $\B$ if
for every partition $P=\setof{p_n}{n\in\omega}$ of $\B$,
the family $\setof{A\sub\omega}{\Join\setof{p_n}{n\in A}\in X}$
of subsets of $\omega$ satisfies $\phi$.
\end{definition}

For some properties $\phi$, the {\em coherent\/} $\phi$
is actually no stronger than $\phi$ itself.
As an easy example, any antichain in $\B$ is a coherent antichain;
and any filter $\F$ on $\B$ is a coherent filter,
as for every partition $P$ of $\B$, the family
$\setof{A\sub\omega}{\Join\setof{p_n}{n\in A}\in \F}$
is a filter on $\omega$. Similarly,
every ultrafilter on $\B$ is a coherent ultrafilter,
and an ultrafilter that is coherently trivial is
a generic ultrafilter on $\B$.
We are interested in ultrafilters with special properties,
for which the coherent version becomes nontrivial.

\medskip
It can be seen from the very definition that the ZFC implications
between various classes of ultrafilters on $\omega$ continue
to hold for the corresponding classes of coherent ultrafilters on $\B$.
For instance, every coherent selective ultrafilter on $\B$ is a coherent
$P$-ultrafilter on $\B$, as every selective ultrafilter on $\omega$
is a $P$-ultrafilter on $\omega$.
%Below, we will study the subclasses of
%coherent $P$-ultrafilters,
%coherent selective ultrafilters,
%coherent $P$-ideals,
%coherent meager ideals,
%and coherent linked systems.
%%FIXME
%\rem{ Dalsi tradicni tridy: tall ideals, Q-filtry, ... }

\subsection{Coherent P-ultrafilters}

%First let us recast the general definition into the special case
%when the property $\phi$ is {\em to be a $P$-point\/}.

\begin{definition}
An ultrafilter $\U$ on a complete ccc algebra $\B$
is a {\em coherent $P$-ultrafilter\/}
if for every partition $P$ of $\B$, the family
$\setof{A\sub\omega}{\Join\setof{p_n}{n\in A}\in \U}$
is a $P$-ultrafilter on $\omega$
\end{definition}

Seeing that the subalgebra $\B_P$ is a copy of $P(\omega)$,
we can equivalently characterize coherent $P$-ultrafilters as follows.

\begin{observation}\label{coherentP}
Let $\B$ be a complete ccc algebra. An ultrafilter $\U$ on $\B$ is
a coherent $P$-ultrafilter iff for every pair of partitions $P$ and $Q$
of $\B$ such that $P \preceq Q$, either $\U \cap Q \neq \emptyset$,
or there is a set $X \sub P$ such that $\Join{X} \in \U$ and for every
$q\in Q$, the set $\setof{p\in X}{p \meet q \neq 0}$ is finite.
\end{observation}

% FIXME
%\rem{
%Prenest dalsi ekvivalenty: rozkladovy mame, co pseudoprunik?
%Ktere dalsi dokazeme prenest? minimalni v RK/RB?
%}

It should probably be noted explicitly that
as the $P$-point condition is only evaluated in the subalgebras $\B_P$,
a coherent $P$-ultrafilter on $\B$ is not, topologically,
a $P$-point in the Stone space of $\B$ ---
unless $\B$ happens to be $P(\omega)$ itself.

\medskip
We show now that coherent $P$-points consistently exist.
The proof is an iteration of the Ketonen argument of \cite{K}
for the existence of $P$-points on $\omega$.

\begin{proposition}
Let $\B$ be a complete ccc Boolean algebra of size at most $\c$.
Every filter on $\B$ with a base smaller than $\c$
can be extended to a coherent $P$-ultrafilter on $\B$
if and only if $\c = \d$.
\end{proposition}

\begin{proof}
Assume $\c = \d$ and let $\F\sub\B$ be a filter with a base smaller than $\c$.
We will construct an increasing chain of filters $\F_\alpha$ extending $\F$,
eventually arriving at a filter $\Union \F_\alpha$, where each $\F_\alpha$
takes care of a pair of partitions, as per \ref{coherentP}.

Start with $\F_0 = \F$ and enumerate all partition pairs
$P \preceq Q$ as $(P_\alpha, Q_\alpha)$, where $\alpha < \d$
runs through all isolated ordinals.
If an increasing chain $\seqof{\F_\beta}{\beta < \alpha}$
of filters has already been found such that every $\F_\beta$
has a base smaller than $\c$ and has the $P$-ultrafilter property
\ref{coherentP} with respect to the partition pairs
$P_\gamma \preceq Q_\gamma$ for $\gamma < \beta$, proceed as follows.

If $\alpha$ is a limit, take for $\F_\alpha$ the filter generated by
$\Union\setof{\F_\beta}{\beta<\alpha}$; then $\F_\alpha$ still has
a base smaller than $\c=\d$. We didn't miss a partition pair here.

If $\alpha = \beta+1$ is a successor,
consider the partition pair $P_\beta \preceq Q_\beta$.
If some $q\in Q_\beta$ is compatible with $\F_\beta$,
let $\F_\alpha = \F_{\beta+1}$ be the filter generated by $\F_\beta \cup \{q\}$
and be done with $(P_\beta, Q_\beta)$.
If there is no such $q$ in $Q_\beta$, enumerate $Q_\beta$ as
$\setof{q_n}{n\in\omega}$ and consider the refinement $P_\beta$
of $Q_\beta$. Without loss of generality, every $q_n\in Q_\beta$ is
partitioned into infinitely many $p\in P_\beta$; %(If not, consider
%$a = \Join\setof{q\in Q}{q\hbox{ is only finitely partitioned by P}}$.
%If $a$ is compatible with $\F_\beta$, take $\F_{\beta+1}$ to be the filter
%generated by $\F_\beta \cup \{a\}$; if not, continue
%working under $-a$ where every $q\in Q$ is infinitely partitioned by $P$).
enumerate $\setof{p\in P}{p < q_n}$ as $\setof{p_n^m}{m\in\omega}$.
Let $\setof{a_\xi}{\xi < \kappa}$ be a base of $\F_\beta$,
for some $\kappa < \c$.
%Every $a_\xi$ intersects infinitely many $q_n$:
%if $a_\xi$ only meets $q_1, \dots, q_n \in Q$,
%choose $a_{\xi_i}$ disjoint with $q_i$, respectively;
%then $a_\xi \leq \Join{q_i}$ is disjoint with $\Meet{a_\xi^i}$
%-- a contradiction.

Now perform the Ketonen construction for this step:
for each $\xi < \kappa$, put
$f_\xi(n) = \min\setof{m}{a_\xi \meet p_n^m \neq 0}$
if there is such an $m$.
The value of $f_\xi(n)$
is defined for infinitely many $n$, corresponding to those $q_n$ which
$a_\xi$ meets. In the missing places, fill the value of $f_\xi(n)$
with the {\it next\/} defined value (there must be some).
This yields a family
$\setof{f_\xi:\omega\to\omega}{\xi < \kappa}$ of functions 
--- which cannot be dominating, as $\kappa < \c = \d$. Therefore, there
is a function $f:\omega\to\omega$ which is not dominated by any $f_\xi$;
that is, for each $\xi$, we have $f(n) > f_\xi(n)$ for infinitely many $n$.
We can assume that $f$ is strictly increasing.

Put $a = \Join\setof{p_n^m}{n\in\omega, m \leq f(n)}$. The element $a$
is compatible with $\F_\beta$, because it meets every $a_\xi$,
as witnessed by $f\not\leq f_\xi$. Let $\F_\alpha$ be the filter generated
by $\F_\beta \cup \{a\}$. This filter
obviously extends $\F_\beta$, is generated by fewer than $\c$ elements,
and has the $P$-ultrafilter property with respect to $(P_\beta, Q_\beta)$.

Now every ultrafilter extending $\Union\setof{\F_\alpha}{\alpha < \c}$
is a coherent $P$-ultrafilter on $\B$ that extends $\F$, because we have
taken care of all possible partition pairs $P\preceq Q$, as requested
by \ref{coherentP}.

The other direction follows from \cite{K} immediately.
Being able to extend every small filter $\F\sub\B$
into a coherent $P$-ultrafilter is apparently stronger than
being able to extend every small filter $\F$ on $\omega$ to a $P$-point,
which itself implies $\c = \d$.
\end{proof}

For completeness, we translate the Ketonen argument
for the opposite direction into the algebra $\B$,
showing how $\d < \c$ can break the coherence {\em anywhere\/}.

Assume $\d < \c$ and let $\setof{f_\alpha}{\alpha < \d}$
be a dominating family of functions.
Choose any two countable partitions $P \preceq Q$ of $\B$
such that every $q_n \in Q$ is partitioned into countably many $p_n^m \in P$.
For each $\alpha < \d$, put $a_\alpha = \Union\setof{p_n^m}{m > f_\alpha(n)}$.
The family $\setof{a_\alpha}{\alpha<\d} \cup \setof{-q_n}{n\in\omega} \sub \B$
is centered, and the filter $\F$ that it generates has $\d < \c$ generators.
No ultrafilter on $\B$ that extends $\F$ can be a coherent $P$-ultrafilter,
as witnessed by $P\preceq Q$.

\medskip
We have shown that coherent $P$-ultrafilters consistently exist
on complete ccc algebras of size not exceeding the continuum.
On the other hand, there consistently is no coherent $P$-ultrafilter
on any complete ccc algebra, as even the classical $P$-points
need not exist \cite{W}.
Hence the existence of coherent $P$-ultrafilters is undecidable in ZFC.

\begin{question}
The consistency we have shown is what \cite{C} calls
``generic existence'' --- under our assumptions,
coherent $P$-ultrafilters not only exist, but
every small filter can be enlarged into one.
Questions arise:
\begin{itemize}
\item [(a)]
Is it consistent that $P$-points exist on $\omega$,
but there are no coherent $P$-ultrafilters
on complete atomless ccc algebras?
\item [(b)]
Is it consistent that a coherent $P$-ultrafilter exists
on a complete atomless ccc algebra $\B$, but does not exist on another?
\item [(c)]
Is there a single ``testing'' algebra $\B$ with the property that
if there is a coherent $P$-ultrafilter on $\B$, then necessarily $\c=\d$,
and hence $P$-ultrafilters exist generically?
\end{itemize}
\end{question}

\subsection{Coherent selective ultrafilters}

Similarly to coherent $P$-ultrafilters,
we start with the following characterization
of coherent selective ultrafilters via partitions.

\begin{observation}\label{coherentS}
Let $\B$ be a complete ccc algebra. An ultrafilter $\U$ on $\B$ is
a coherent selective ultrafilter iff for every pair of partitions $P$ and $Q$
of $\B$ such that $P \preceq Q$, either $\U \cap Q \neq \emptyset$,
or there is a set $X \sub P$ such that $\Join{X} \in \U$ and for every
$q\in Q$, the set $\setof{p\in X}{p \meet q \neq 0}$ is at most a singleton.
\end{observation}

The following proposition generalizes the arguments from [Ke] and [Ca]
on existence of selective  ultrafilters on $\omega$ to coherent selective
ultrafilters on complete ccc algebras.

\begin{proposition}
Let $\B$ be a complete ccc Boolean algebra of size at most $\c$.
Then every filter $\F$ on $\B$ with a base smaller than $\c$
can be extended to a coherent selective ultrafilter on $\B$
if and only if $\c = \cov(\cal M)$.
\end{proposition}

\begin{proof}
Assume $\c = \cov(\cal M)$ and let $\F$ be a filter with a base
smaller than $\c$. We will construct an increasing chain of filters
extending $\F$.
Put $\F_0 = \F$ and enumerate all partition pairs $P\preceq Q$
as $\setof{(P_\alpha, Q_\alpha)}{\alpha < \cov(\cal M)\hbox{ isolated }}$.

If an increasing chain $\seqof{\F_\beta}{\beta < \alpha}$
of filters has been found such that every $\F_\beta$ has a base smaller
than $\c$ and has the selective property with respect to all
$\setof{(P_\gamma, Q_\gamma)}{\gamma < \beta}$, proceed as follows.

If $\alpha$ is a limit, take for $\F_\alpha$ the filter generated by
$\Union\setof{\F_\beta}{\beta<\alpha}$; then $\F_\alpha$ still has
a base smaller than $\c$.

If $\alpha=\beta+1$ is a successor, consider $(P,Q)=(P_\beta, Q_\beta)$.
Without loss of generality, both partitions are infinite,
and every $q_n\in Q$ is infinitely partitioned into $p_n^m\in P$.

If there is some $q\in Q$ compatible with $\F_\beta$,
let $\F_\alpha$ be the filter generated by $\F_\beta \cup \{q\}$.
If there is no such $q\in Q$, consider some base
$\setof{a_\xi}{\xi < \kappa}$ of $\F_\beta$, where $\kappa < \c$.
Every $a_\xi$ intersects infinitely many $q\in Q$:
if $a_\xi$ only meets $q_1, \dots, q_n \in Q$,
choose $a_\xi^i$ disjoint with $q_i$, respectively;
then $a_\xi \leq \Join{q_i}$ is disjoint with $\Meet{a_\xi^i}$
--- a contradiction.

Consider the set $T = \Pi_{n\in\omega}\setof{p_n^m}{m\in\omega}$;
the functions $\phi\in T$ are the selectors for $Q$.
View $T$ as a copy of the Baire space $\omega^\omega$.
If no selector for $Q$ %$a\in\B_P$ for the partition $Q$
%(i.e., an element $a = \Join{X}\in\U_P$ for some $X\sub P$ such that
%for every $q\in Q$, the set $\setof{p\in X}{p \meet q}$ is a singleton)
is compatible with $\F_\beta$, put $T_\xi=\setof{\phi\in T}{
\Join{rng(\phi)}\perp a_\xi}$; then we have $T=\Union_{\xi<\kappa}{T_\xi}$.
But the sets $T_\xi$ cannot cover $T$, as $\kappa < \cov(\cal M)$
and every $T_\xi$ is a nowhere dense subset of $T$, which is seen as follows.

For a basic clopen subset $[s]$ of $T$,
there is some $n>|s|$ such that $a_\xi$ meets $q_n\in Q$,
because $a_\xi$ meets infinitely many $q_n$.
Hence some $p_n^m$ meets $a_\xi$.
Extend $s$ into $t$ so that $t(n) = m$.
Then $[t]\sub[s]$ is disjoint with $T_\xi$.

Thus there must be a selector $\phi\in T$ with $ b = \Join{rng(\phi)}$
compatible with every $a_\xi$. Let $\F_{\beta+1}$ be the filter generated
by $\F_\beta \cup \{b\}$. Iterating this process, we obtain an increasing
sequence of filters $\seqof{\F_\alpha}{\alpha\in\c}$ extending $\F=\F_0$.
Now every ultrafilter extending $\Union\F_\alpha$
is a coherent selective ultrafilter on $\B$
by \ref{coherentS}.
\end{proof}

\section{Nonhomogeneity}\label{homogeneity}

\begin{definition}
A topological space $X$ is {\em homogeneous\/}
if for every pair of points $x,y \in X$
there is an autohomeomorphism $h$ of $X$ such that $h(x) = y$.
\end{definition}

Extremally disconnected compact Hausdorff spaces,
which are precisely the Stone spaces of complete Boolean algebras,
are long known {\em not\/} to be homogeneous.
However, the original elegant proof due to Frol{\'\i}k \cite{F}
suggests no simple topological property of points
to be a reason for this.

If a space $X$ is not homogeneous, then points $x,y \in X$
failing the automorphism property are often called
{\em witnesses of nonhomogeneity\/}.
In large subclasses of the extremally disconnected compacts,
such witnesses have been found by isolating a simple topological
property that is shared by some, but not all, points in the space.

\begin{definition}
A point $x$ of a topological space $X$ is
and {\em untouchable point\/} if $x\notin \overline D$
for every countable nowhere dense subset $D \sub X$
not containing $x$.
\end{definition}

The subclass of extremally disconnected compact spaces
where a witness of nonhomogeneity hasn't been explicitly described yet
is currently reduced to the class of ccc spaces of weight at most continuum.
In other extremally disconnected compacts,
points with even stronger properties have been found.
See \cite{S} and \cite{BS} for history and pointers
to the development of these questions.

\subsection{An application to nonhomogeneity}

Via Stone duality, the topic has a Boolean translation:
we are looking for discretely untouchable ultrafilters
on complete ccc Boolean algebras of size
(or, equivalently, algebraic density) at most continuum.
It is in this form that we actually deal with the question.

%Now we show the relevance of coherent $P$-ultrafilters
%to the Simon Conjecture: they provide a consistent positive answer.

\begin{proposition}
Let $\B$ be a complete ccc algebra.
Let $\U$ be a coherent $P$-ultrafilter on $\B$.
Then $\U$ is an untouchable point in %$\St(\B)$.
the Stone space of $\B$.
\end{proposition}

\begin{proof}
We assume that $\U$ is not an atom, otherwise there is nothing to prove.
Let $R = \setof{\F_n}{n\in\omega}$ be a~countable nowhere dense set in
$\St(\B)$ such that $\F_n \not= \U$ for all $n$. Choose some $a_0 \in \F_0$
with $-a_0 \in \U$ and put $R_0 = \setof{\F\in R}{a_0 \in \F} \sub R$.
Generally, if $a_i\in \B^+$ for $i < k$ are disjoint elements such that
$\Join_{i<k}{a_i} \notin \U$ and $R_i = \setof{\F\in R}{a_i\in\F}$,
consider $\Union_{i<k}{R_i}\sub R$. If $\Union_{i<k}{R_i} = R$, we are done,
as $-\Join_{i<k}{a_i}\in \U$ guarantees $\U \not \in cl(R)$. Otherwise, let
$n_k$ be the first index such that $\F_{n_k} \not \in \Union_{i<k}{R_i}$ and
choose some $a_k$ disjoint with $\Join_{i<k}{a_i}$ such that $a_k\in\F_{n_k}$
and $a_k \notin \U$.

This construction either stops at some $k$ and we are
done, or we arrive at an infinite disjoint system
$Q=\setof{a_i}{i\in\omega}\sub\B^+$.
Again, if $\Join{Q}\not \in \U$, we have $\U \not \in cl(R)$.
Otherwise, we can assume that $\Join{Q} = 1$, so $Q$ is a partition of $\B$.
For each $a_i \in Q$, choose an infinite partition $P_i$ of $a_i$ such that
$P_i \cap \Union{R_i} = \emptyset$ -- this is possible, because $R_i \sub R$
is nowhere dense. Now $P=\Union{P_i} \preceq Q$ is a partition pair in $\B$.

As $\U$ is a coherent $P$-ultrafilter and misses $Q$,
there is some $X\sub P$ with $u = \Join{X}\in \U$
such that for every $i$, the set $\setof{p\in X}{p\leq a_i}$ is finite.
This means that $u\notin\F_n$ for all $n$: every $\F_n$ is in one
particular $a_i$, so $u \in \F_n$ would mean that $\F_n$ contains
one of the finitely many $\setof{p \leq u}{p\leq a_i}$. But this is
in contradiction with $P_i \cap \Union{R_i} = \emptyset$.
So $u\in\U$ isolates $\U$ from $cl(R)$.
\end{proof}

In fact, we have proven a slightly stronger statement:
$\U$ escapes the closure of any nowhere dense set
that can be covered by countably many disjoint open sets.

%\bigskip
%We note in closing that the result above cannot be
%an ultimate solution to the nonhomogeneity question,
%as we had to venture out of ZFC for existence.

\bibliographystyle{plain}

\begin{thebibliography}{BJST}
\bibitem[BS]{BS}
	B.~Balcar, P.~Simon,
	{\em On minimal $\pi$-character of points in
	extremally disconnected compact spaces},
	Topology Appl. 41 (1991), 133--145. 
\bibitem[C]{C}
	R.~M.~Canjar,
	{\em On the generic existence of special ultrafilters},
	Proc. AMS 110:1 (1990), 233--241
\bibitem[F]{F}
	Z.~Frol{\'\i}k,
	{\em Maps of extremally disconnected spaces,
		theory of types, and applications},
		in: Franklin, Frol{\'\i}k, Koutn{\'\i}k (eds.):
	{\em General Topology and Its Relations to Modern Analysis and Algebra.
	Proceedings of the Kanpur topological conference} (1971),
	131--142
\bibitem[K]{K}
	J.~Ketonen,
	{\em On the existence of $P$-points
	in the Stone-\v{C}ech compactification of integers},
	Fund. Math. 92 (1976), 91--94
\bibitem[S]{S}
	P.~Simon,
	{\em Points in extremally disconnected compact spaces},
	Rend. Circ. Mat. Palermo (2). Suppl. 24 (1990),
	203--213. 
\bibitem[W]{W}
	E. L. Wimmer,
	{\em The Shelah $P$-point Independence Theorem},
	Israel J. Math 43:1 (1982), 28--48

\end{thebibliography}

\end{document}